\documentclass[12pt,a4paper]{article}

\usepackage{graphicx}
\usepackage{amsmath}
\usepackage{amssymb}
\usepackage{amscd}
\usepackage{amsthm}
\usepackage{cleveref}
\usepackage[noadjust]{cite}
\usepackage{enumitem}

\newtheorem{theorem}{Theorem}[section]
\newtheorem{corollary}{Corollary}

\newtheorem{lemma}[theorem]{Lemma}
\newtheorem{proposition}{Proposition}

\theoremstyle{definition}
\newtheorem{definition}[theorem]{Definition}
\newtheorem{remark}{Remark}
\newcommand{\EQ}{\begin{equation}}
\newcommand{\EN}{\end{equation}}

\newcommand{\zero}{{\mathbf{0}}}

\newcommand{\bh}{{\bf h}}
\newcommand{\bc}{{\bf c}}

\newcommand{\br}{{\bf r}}

\newcommand{\by}{{\bf y}}
\newcommand{\bx}{{\bf x}}

\newcommand{\bv}{{\bf v}}

\newcommand{\F}{\mathbb{F}}

\newcommand{\BB}{{\mathcal{B}}}
\newcommand{\DD}{{\mathcal{D}}}

\newcommand{\IA}{\operatorname{IA}}
\newcommand{\Gcd}{\operatorname{gcd}}
\newcommand{\Aut}{\operatorname{Aut}}

\title{Completely regular codes by concatenating Hamming codes}

\author{J. Borges, J. Rif\`{a}\footnote{email:~josep.rifa@autonoma.edu, joaquim.borges@autonoma.edu} \\
Department of Information and Communications Engineering,\\
 Universitat Aut\`{o}noma de Barcelona,\\
\and V. A. Zinoviev\footnote{e-mail:\, zinov@iitp.ru}\\A.A. Kharkevich Institute for Problems of Information
Transmission,\\ Russian Academy of Sciences}

\begin{document}
\maketitle

%
%
%
%


\begin{abstract}
We construct new families of completely regular codes by
concatenation methods. By combining parity check matrices of cyclic Hamming codes, we obtain families of completely regular codes. In all cases, we compute the intersection array of these codes. We also study when the extension of these codes gives completely regular codes.
Some of these new codes are
completely transitive.
\end{abstract}

\section{Introduction}
Let $\F_q$ be a finite field of the order $q$. A $q$-ary linear $[n,k,d;\rho]_q$-code $C$ is a
$k$-dimensional subspace of $\F_q^n$, where $n$ is the {\em
length}, $d$ is the {\em minimum distance}, $q^k$ is the
{\em cardinality} of $C$, and $\rho$ is the {\em covering radius}. For $q=2$, we omit the subscript $q$. The {\em packing radius} of $C$ is $e=\lfloor (d-1)/2 \rfloor$. Given any vector $\bv \in \F_q^n$, its
{\em distance to the code $C$} is $d(\bv,C)=\min_{\bx \in C}\{
d(\bv, \bx)\}$ and the covering radius of the code $C$ is
$\rho=\max_{\bv \in \F_q^n} \{d(\bv, C)\}$. Note that $e\leq \rho$. We denote by $~D=C+\bx~$ a
{\em coset} of  $C$, where $+$ means the component-wise addition
in $\F_q$.

For a given $q$-ary code $C$ of length $n$ and covering radius $\rho$,
define
\[
C(i)~=~\{\bx \in \F_q^n:\;d(\bx,C)=i\},\;\;i=0,1,\ldots,\rho.
\]
The sets $C(0)=C,C(1),\ldots,C(\rho)$ are called the {\em subconstituents} of $C$.

Say that two vectors $\bx$ and $\by$ are {\em neighbors} if
$d(\bx,\by)=1$. Denote by $\zero$ the all-zero vector.

\begin{definition}[\cite{Neum}]\label{de:1.1} A $q$-ary code $C$ of length $n$ and covering radius $\rho$ is {\em completely regular}, if
for all $l\geq 0$ every vector $x \in C(l)$ has the same number
$c_l$ of neighbors in $C(l-1)$ and the same number $b_l$ of
neighbors in $C(l+1)$. Define $a_l = (q-1){\cdot}n-b_l-c_l$ and
set $c_0=b_\rho=0$. The parameters $a_l$, $b_l$ and $c_l$ ($0\leq l\leq \rho$) are called {\em intersection numbers} and the sequence $\{b_0, \ldots, b_{\rho-1}; c_1,\ldots,
c_{\rho}\}$ is called the {\em intersection array} (shortly $\IA$) of $C$.
\end{definition}

Let $M$ be a monomial matrix, i.e. a matrix with exactly one
nonzero entry in each row and column. If $q$ is prime, then
the automorphism group of $C$, $\Aut(C)$, consists of all monomial ($n\times n$)-matrices $M$ over
$\F_q$ such that $\bc M \in C$ for all $\bc \in C$. If $q$ is a
power of a prime number, then $\Aut(C)$ also contains any field
automorphism of $\F_q$ which preserves $C$. The group $\Aut(C)$
acts on the set of cosets of $C$ in the following way: for all
$\pi\in \Aut(C)$ and for every vector $\bv \in \F_q^n$ we have
$\pi(\bv + C) = \pi(\bv) + C$.

\begin{definition}[\cite{Giu,Sole}]\label{de:1.3}
Let $C$ be a linear code over $\F_q$ with covering radius $\rho$.
Then $C$ is {\em completely transitive} if $\Aut(C)$ has $\rho +1$ orbits
when acts on the cosets of $C$.
\end{definition}

Since two cosets in the same orbit have the same weight
distribution, it is clear that any completely transitive code is
completely regular.

Completely regular and completely transitive codes are classical
subjects in algebraic coding theory, which are closely connected
with graph theory, combinatorial designs and algebraic
combinatorics. Existence, construction and enumeration of all such
codes are open hard problems (see \cite{BCN,Dam,Koo,Neum} and
references there).

It is well known that new completely regular codes can be obtained by direct sum of perfect codes or,
more general, by direct sum of completely regular codes with covering radius $1$ \cite{BZZ,Sole}.
In the current paper, we extend these constructions, giving several explicit
constructions of new completely regular and completely transitive codes, based on concatenation
methods.

\section{Preliminary results}

In this section we see several results we will need in the next sections.

\begin{lemma}[\protect{\cite{Neum}}]\label{graph}
Let $C$ be a completely regular code with covering radius $\rho$ and intersection array $\{b_0,\ldots,b_{\rho-1};c_1,\ldots,c_\rho\}$.
If $C(i)$ and $C(i+1)$, $0\leq i <\rho$, are two subconstituents of $C$, then
$$
b_i|C(i)|=c_{i+1}|C(i+1)|.
$$
\end{lemma}


\begin{definition}[\protect{\cite{GvT}}]\label{def:2.5}
A quasi-perfect $e$-error-correcting $q$-ary code $C$ is called
{\em uniformly packed} if there exist natural numbers $\lambda$ and $\mu$
such that for any vector $x$:
$$
B_{x,e+1} = \left\{
\begin{array}{cl}
  \lambda & \mbox{ if } d(x,C)=e, \\
  \mu     & \mbox{ if } d(x,C)=e+1.
\end{array}\right.
$$
\end{definition}

Van Tilborg \cite{vTi} (see also \cite{Lind,SZZ})
showed that no nontrivial codes of this kind exist for
$e>3$.

\begin{proposition}[\protect{\cite{GvT}}, see also \cite{SZZ}]
A uniformly packed code is completely regular.
\end{proposition}

\begin{definition}
A $t$-$(v,k,\lambda)$-{\em design} is an incidence structure $(S,\BB)$,
where $S$ is a $v$-set of elements (called {\em points}) and $\BB$ is
a collection of $k$-subsets of points (called {\em blocks})
such that every $t$-subset
of points is contained in exactly $\lambda>0$ blocks ($0< t\leq
k\leq v$).
\end{definition}

In terms of incident matrix a $t$-$(v,k,\lambda)$-design is a binary code $C$
of length $v$ with codewords of weight $k$ such that any binary vector
of length $v$ and weight $t$ is covered by exactly $\lambda$ codewords.
A $t$-design with $\lambda=1$ is called a Steiner system and also
denoted by $S(v,k,t)$.
The following properties are well known (e.g., \cite{Beth,Blake,Hugh}).

\begin{proposition}\label{blocks}
Given a $t$-$(v,k,\lambda)$-design, every $i$-subset of points ($0\leq
i\leq t$) is contained in exactly $\lambda_i$ blocks, where
$$
\lambda_i=\lambda\frac{\binom{v-i}{t-i}}{\binom{k-i}{t-i}}.
$$
\end{proposition}

\begin{corollary}\label{cordissenys}
Given a $t$-$(v,k,\lambda)$-design $\DD$:
\begin{itemize}
\item[(i)] $\DD$ is an
$i$-design, for all $i\leq t$.
\item[(ii)] $\lambda=\lambda_t$.
\item[(iii)] The
number of blocks of $\DD$ is $b=\lambda_0$.
\item[(iv)] Each point
is contained in the same number of blocks, namely
$r=\lambda_1=bk/v$ ($r$ is called the {\em replication number}).
\end{itemize}
\end{corollary}

There is a natural $q$-ary generalization of such $t$-designs (see \cite{AssGM,Del,GvT,Rif1}).
Let $E=\{0,1, \ldots, q-1\}$. A collection $\BB$ of $b$
vectors $x_1, \ldots,x_b$ of length $v$ and weight $k$ over $E$ is
called a $q$-ary $t$-design and denoted $t$-$(v,k,\lambda)_q$, if for every
vector $y$ over $E$ of length $v$ and weight $t$ there are exactly $\lambda$
vectors $x_{i_1}, \ldots, x_{i_\lambda}$ from $\BB$ such that
$d(y,x_{i_j}) = k-t$ for all $j=1,\ldots,\lambda$. If $\lambda=1$,
then we obtain a $q$-ary Steiner system, denoted $S(v,k,t)_q$.

For a code $C$ denote by $C_w$ the set of all codewords of $C$ of weight $w$.
Regularity of a code $C$ implies that the sets $C_w$ determine $t$-designs.

Directly from the definition of completely regular codes (see also \cite{GvT,SZZ}) we have the following

\begin{theorem}\label{designs}
Let $C$ be a $q$-ary completely regular code of length $n$ with minimum distance $d$.

\begin{itemize}
\item[(i)] If $d=2e+1$ then any nonempty set $C_w$ is an
$e$-$(n,w,\lambda_w)_q$-design.
\item[(ii)] If $d=2e+2$ then any nonempty
set $C_w$ is an $(e+1)$-$(n,w,\lambda_w)_q$-design.
\end{itemize}
\end{theorem}

For a code $C$, we denote by $s+1$ the number of nonzero terms in the dual
distance distribution of $C$, obtained by the MacWilliams transform. The parameter
$s$ was called {\em external distance} by Delsarte \cite{Del}, and is equal
to the number of nonzero weights of $C^\perp$ if $C$ is linear. The following properties show the importance of this parameter.

\begin{theorem}\label{params}
If $C$ is any code with covering radius $\rho$ and external distance $s$, then
\begin{itemize}
\item[(i)] {\cite{Del}} $\rho \leq s$.
\item[(ii)] {\cite{Del}} A code $C$ is perfect ($e=\rho$) if and only if $e=s$.
\item[(iii)] {\cite{GvT}} A code $C$ is quasi-perfect uniformly packed if and only if $s=e+1$.
\item[(iv)] {\cite{Sole}} If $C$ is completely regular, then $\rho=s$.
\item[(v)] {\cite{Del}} If $d\geq 2 s -1$, then $C$ is completely regular.
\item[(vi)] {\cite{BCN}} If $C$ has only even weights and $d\geq 2s -2$, then $C$ is completely regular.
\end{itemize}
\end{theorem}

Given a code $C$, we define the {\em extended code} $C^*$ by adding an extra coordinate to each codeword of $C$ such that the sum of the coordinates of the extended vector is zero.

\begin{proposition}\label{estesos}
If a binary extended code $C^*$, of length $n+1$, is a completely regular code with minimum distance $d^*=2e+2\geq 4$, then
for all odd $w$
\begin{eqnarray*}\label{recursio}
|C^*_{w+1}|(w+1) &=& (n+1)|C_w|\;\;\mbox{ and }\\
(n-w)|C_w| &=& (w+1)|C_{w+1}|.
\end{eqnarray*}
\end{proposition}

\begin{proof}
Let $w$ be odd and assume that $C_w$ is not empty. By Theorem \ref{designs}, the set
$C^*_{w+1}$ of codewords of weight $w+1$ form a $(e+1)$-$(n+1,w+1,\lambda^*_2)$-design which, in particular,
is a $2$-$(n+1,w+1,\lambda^*_2)$-design, by Corollary \ref{cordissenys}. The number of codewords in $C^*_{w+1}$ with nonzero value at position
$n+1$ is $r^*$, the replication number, and clearly $r^*=|C_w|$. Therefore,
\begin{equation}\label{replication}
|C^*_{w+1}|(w+1)=(n+1)r^*.
\end{equation}
Combining (\ref{replication}) with $|C^*_{w+1}|=|C_{w+1}|+|C_{w}|$, the result follows.
\end{proof}

For any vector $x=(x_1,\ldots,x_n)\in \F_q^n$, denote by $\sigma(x)$ the right cyclic shift of $x$, i.e. $\sigma(x)=(x_n,x_1,\ldots,x_{n-1})$. Define recursively $\sigma^i(x)=\sigma(\sigma^{i-1}(x))$, for $i=2,3,\ldots$ and $\sigma^1(x)=\sigma(x)$. For $j<0$, we define $\sigma^j(x)=\sigma^{\ell}(x)$, where $\ell = j \mod n$.

Finally, we will also make use of the following technical lemma.

\begin{lemma}\label{shifts}
Let $x\in \F_q^n$ be a vector of weight $w$. If $\Gcd(n,w)=1$, then $\sigma^i(x)\neq x$, for all $i=1,\ldots,n-1$.
\end{lemma}

\begin{proof}
Assume that $\sigma^i(x)=x$ for some $i=2,\ldots,n-1$. Then, $i$ divides $n$ and $x$ has the form:
$$
x=\underbrace{(x',x',\ldots,x')}_{n/i},
$$
where $x'$ is a vector of length $i$. Thus, $w$ is a multiple of $i$. As a consequence, $i$ is a common divisor of $n$ and $w$. For the case, $\sigma(x)=x$, note that $x$ should be either the all-one or the all-zero vector.
\end{proof}

\section{Infinite families of CR codes}

The next construction is new, although the dual codes of the resulting family of $q$-ary completely regular codes are known as the family SU2 in \cite{Cald}. In the current paper, we also study when these codes are completely transitive and when the extended codes are completely regular.

\subsection*{Construction I}

Let $H$ be the parity check matrix of a $q$-ary cyclic Hamming code of length $n=(q^k-1)/(q-1)$, (hence $\Gcd(n,q-1)=1$). Thus, the simplex code generated by $H$ is also a cyclic code. Denote by $\br_1,\ldots,\br_k$ the rows of $H$. For any $c\in \{2,\ldots,n\}$, consider the code $C$ with parity check matrix
\begin{equation}\label{code33}
\left[
\begin{array}{ccccc}
H\;& \;H \;&\ldots\,&\,H\,\\
H_1\;& \;H_2 \;&\ldots\,&\,H_c
\end{array}
\right],
\end{equation}
where $H_i$ is the matrix $H$ after cyclically shifting $i$ times its columns to the right. In other words, the rows of $H_i$ are $\sigma^i(\br_1),\ldots,\sigma^i(\br_k)$. Note that, for $c=1$, we have
\[
\left[
\begin{array}{c}
H\;\\
H_1\;
\end{array}
\right],
\]
which generates the simplex code as $H$. Therefore, in this case, $C$ is a Hamming code.

\begin{proposition}\label{pesos}
The code $C^\perp$ has nonzero weights
$$
w_1=cq^{k-1}\;\;\mbox{ and }\;\;w_2=(c-1)q^{k-1}.
$$
\end{proposition}

\begin{proof}
Let $x=(x_1,\ldots,x_c)\in C^\perp$ be a nonzero codeword such that each $x_i$ is an vector of length $n$ generated by
\[
\left[
\begin{array}{c}
H\;\\
H_i\;
\end{array}
\right].
\]
Since $H$ and $H_i$ generate the same simplex code, $x_i$ has weight $0$ or $q^{k-1}$. Assume that $x_i$ is the zero vector. Then, $x_i$ is generated by a linear combination of the rows of $H$ (giving some vector $v$), together with a linear combination of the rows of $H_i$ (giving the vector $-v$). The same linear combination of the rows of $H_j$ gives the vector $u=\sigma^{j-i}(-v)$. Since the weight of $-v$ is $q^{k-1}$ and  $\Gcd((q^k-1)/(q-1),q^{k-1})=1$, we have, by Lemma \ref{shifts}, $u\neq -v$ and hence $x_j$ is not the zero vector.

The conclusion is that $x$ has weight $cq^{k-1}$ or $(c-1)q^{k-1}$.
\end{proof}

\begin{remark}
In the proof of Proposition \ref{pesos}, the number of ways to get $x_i$ equal to the zero vector (being $x$ a nonzero codeword) is equal to the number of nonzero vectors generated by $H$. Therefore, $C^\perp$ has $c(q^k-1)$ codewords of weight $w_2=(c-1)q^{k-1}$ and $q^{2k}-c(q^k-1)-1$ codewords of weight $w_1=cq^{k-1}$. By using this weight distribution of $C^\perp$ and the MacWilliams transform \cite{MacW}, it is possible to compute $|C_3|$, the number of codewords in $C$ of weight $3$. Here we use a combinatorial argument to compute $|C_3|$.
\end{remark}

Let $B_1,\ldots,B_c$ be the $n$-sets, which we call {\em blocks}, of coordinate positions corresponding to $H_1,\ldots,H_c$, that is $B_j=\{(j-1)n+1,(j-1)n+2,\ldots,jn\}$, for $j=1,\ldots,c$.

\begin{proposition}\label{pesminim}
The number $|C_3|$ of codewords in $C$ of weight $3$ is:
$$
(q-1)^2\frac{c\binom{n}{2}}{3} + (q-1)n\binom{c}{3}=\frac{(q-1)cn}{6}\left[(q-1)(n-1)+(c-1)(c-2)\right],
$$
provided that $\binom{c}{3}=0$, for $c\in\{1,2\}$.
\end{proposition}

\begin{proof}
For $c=1$, the result is trivial since
$$
(q-1)^2\frac{\binom{n}{2}}{3}=\frac{(q-1)^2n(n-1)}{6}
$$
is the number of triples in a $q$-ary $2$-$(n,3,1)$-design. Note that any codeword $x$ of weight $3$ cannot have exactly $2$ nonzero coordinates in the same block because there exists a codeword $y$ in such block covering these two coordinates and, hence, we would have $d(x,y)=2$. Thus, the result is also trivial for $c=2$.

If $c>2$, then the codewords of weight $3$ are divided into two classes: a) those with the three nonzero coordinates in the same block, and b) those with the three nonzero coordinates in three different blocks.

Clearly, the number of codewords in the case a) is
$$
c(q-1)^2\frac{\binom{n}{2}}{3}.
$$

For the case b), consider any three distinct blocks $B_{j_1}$, $B_{j_2}$, and $B_{j_3}$ (we can choose these three blocks in $\binom{c}{3}$ ways). In the block $B_{j_1}$, we fix a vector $v$ of weight one (we have $(q-1)n$ such vectors). Now, we claim that there exists exactly one codeword of weight $3$ covering $v$ with the other two nonzero coordinates in $B_{j_2}$ and $B_{j_3}$.

If there are two such codewords, say $x=v+e_2+e_3$ and $y=v+d_2+d_3$ ($e_\ell$ and $d_\ell$ are one-weight vectors with the nonzero coordinate in $B_{j_\ell}$, for $\ell=2,3$), then we know that there are $3$-weight codewords $x'$ and $y'$ with nonzero coordinates in $B_{j_2}$ and $B_{j_3}$, respectively, and covering $e_2+e_3$ and $d_2+d_3$, respectively. Therefore, $x+y+x'+y'$ has weight $2$ leading to a contradiction.

If there are not any codeword covering $v$ with nonzero coordinates in $B_{j_2}$ and $B_{j_3}$, then any vector $v+e_2$ is at distance two from $C$. Thus, we can get $(q-1)^2n^2$ vectors $z$, such that $d(z,C)=d(z,\zero)$. We know that $|C(2)|=(q^{2k}-(q-1)nc-1)|C|$, since the covering radius of $C$ is $\rho=2$, and clearly $|C(1)|= (q-1)nc|C|$. Therefore, the number of vectors in $C(2)$ at distance $2$ from the zero codeword is $(q^{2k}-(q-1)nc-1)$. As a consequence, we should have
$$
(q^{2k}-(q-1)nc-1) \geq (q-1)^2n^2,
$$
which gives $(q^k+1)(q^k-1)-(q^k-1)c \geq (q^k-1)^2$, and hence $c\leq 2$,
which contradicts the assumption $c>2$.

The statement is proved.
\end{proof}

\begin{corollary}
The code $C$ with parity check matrix given in (\ref{code33}) is a quasi-perfect uniformly packed code (hence completely regular) with parameters $[nc,nc-2k,3;2]_q$ and intersection array
$$
\IA=\{(q-1)nc,\left((q-1)n-c+2\right)(c-1);1,c(c-1)\}.
$$
\end{corollary}

\begin{proof}
The length, dimension and minimum distance of $C$ are clear. By Proposition \ref{pesos}, $C$ has external distance $s=2$. Since $C$ is not perfect, $1 < \rho$. Thus, by Theorem \ref{params} (i), the covering radius is $\rho=2$, and by Theorem \ref{params} (iii), $C$ is a quasi-perfect uniformly packed code.

The values of the intersection numbers $b_0=(q-1)n$ and $c_1=1$ are straightforward since $C$ has minimum distance $3$.

Now, we compute the intersection number $a_1$, that is, the number of neighbors in $C(1)$ of any vector $z\in C(1)$. Without loss of generality, assume that $z$ is a one-weight vector. Then, $a_1$ is the addition of the number of two-weight vectors covering $z$ and covered by some codeword of weight $3$, and the $q-2$ vectors of weight $1$ at distance $1$ from $z$. Since the set $C_3$ of codewords of weight $3$ defines a $q$-ary 1-design (Theorem \ref{designs}), we have that
\begin{equation}\label{disseny}
3|C_3|=(q-1)cnr,
\end{equation}
where $r$ is the replication number, i.e. the number of codewords in $C_3$ covering $z$ (note that (\ref{disseny}) is a generalization to the $q$-ary case of Corollary \ref{cordissenys} (iv)). Of course, any such codeword covers two vectors of weight $2$ that, also, cover $z$. Thus, we have that $a_1=2r+q-2$. Combining with (\ref{disseny}), we obtain
$$
a_1=\frac{6|C_3|}{(q-1)cn}+q-2,
$$
and substituting $|C_3|$ from Proposition \ref{pesminim}, we get
$$
a_1=\left[(q-1)(n-1)+(c-1)(c-2)\right]+q-2.
$$
Since $b_1=(q-1)cn-c_1-a_1$, we obtain
\begin{eqnarray*}
b_1 &=& (q-1)cn-1-\left[(q-1)(n-1)+(c-1)(c-2)\right]-(q-2)\\
    &=& ((q-1)n-c+2)(c-1).
\end{eqnarray*}

By Lemma \ref{graph},
we have that $b_1|C(1)|=c_2|C(2)|$. Since $C$ has minimum distance $3$, we have $|C(1)|=(q-1)cn|C|$. Also, $|C(2)|=q^{cn}-|C(1)|-|C|$ because the covering radius of $C$ is $\rho=2$. Therefore, we can compute $c_2$:

$$
c_2 = \frac{b_1|C(1)|}{|C(2)|}=\frac{((q-1)n-c+2)(c-1)(q-1)cn|C|}{(q^{2k}-(q-1)nc-1)|C|}.
$$
Substituting $n=(q^k-1)/(q-1)$, the expression simplifies to $c(c-1)$. This completes the proof.
\end{proof}

\begin{remark}
Almost all codes in the family described in Construction I
are not completely transitive codes. However, software computations
suggest that
in the binary case and for any value of $k$ (so $n=2^k-1$), the completely transitive
codes of that family are those with $c\in\{2,3,n-1,n\}$. In
general, in the $q$-ary case when $q$ is a power of two, the completely transitive
codes are those with $c\in \{2,3\}$ and if $q=p^r$, for $p\not=2$,
then the completely transitive codes are those with $c=2$.
\end{remark}

\begin{remark}
By extending the codes in the family given in Construction I we
do not obtain completely regular codes, except for the binary case
when the parameter $c$ equals $2^{k-1}+1$. In this case, the
family of extended $[n(2^{k-1}+1)+1,n(2^{k-1}+1)-2k,4;3]$ codes we
obtain coincides with the family described in Theorem
\ref{estes}.
\end{remark}

\subsection*{Construction II}

The next construction  works again for $q$-ary cyclic Hamming $[n,k,3;1]_q$ codes, where $n=(q^k-1)/(q-1)$ and $\Gcd(n,q-1)=1$. For a given such
code of length $n$ with parity check matrix $H$, the matrices
$H_i$, \; $i = 1,2, \ldots, n-1$ are defined as in Construction I.
Let $c$ be any integer from the range: $1 \leq c \leq
n-1$ and let $C$ be the code with parity check matrix

\[
H^{(c)} = \left[
\begin{array}{ccccc}
H\;& \;0 \;&\,H\,&\,H\,\ldots \,&\,H\\
0\;& \;H \;&\,H\,&\,H_1\,\ldots \,&\,H_c
\end{array}
\right]
\]
where $0$ denotes the zero matrix (of the same size as $H$).

\begin{proposition}\label{pesos2}
The code $C^\perp$ has nonzero weights
$$
w_1=(c+3)q^{k-1}\;\;\mbox{ and }\;\;w_2=(c+2)q^{k-1},
$$
except if $c=n-1$ and $q=2$. In this case, $C^\perp$ has only the nonzero weight $w=2^{2k-1}$ and $C$ is a Hamming code of length $2^{2k}-1$.
\end{proposition}

\begin{proof}
Assume that $c< n-1$ or $q > 2$. As in Construction I, let $B_1,\ldots,B_{c+3}$ be the sets (blocks) of consecutive $n$ coordinate positions. Note that any linear combination of the first (respectively, second) $k$ rows of $H^{(c)}$ gives a codeword in $C^\perp$ with the zero vector in, and only in, the block $B_2$ (respectively, $B_1$). For a linear combination which gives nonzero vectors in $B_1$ and $B_2$, we have that
the obtained codeword in $C^\perp$ can have the zero vector in at most one block $B_j$, for $j=3,\ldots,c$. This is true by the same argument used in the proof of Proposition \ref{pesos}.

In the binary case, if $c=n-1$, we have that any nonzero codeword in $C^\perp$ has the zero vector in exactly one block. Indeed, any linear combination which gives nonzero vectors in $B_1$ and $B_2$, gives some vector $v+u$, where $v$ is generated by the first $k$ rows of $H^{(c)}$, and $u$ is generated by the second $k$ rows of $H^{(c)}$. Clearly, $v$ and $u$ have the forms:
\begin{eqnarray*}
v &=& \left(y, \zero, y, y,\ldots,y \right), \;\;\mbox{ and }\\
u &=& \left(\zero, x, x, \sigma(x),\sigma^2(x),\ldots,\sigma^{n-1}(x)\right),
\end{eqnarray*}
for some $x,y\in\F_q^n$. Since $x,\sigma(x),\sigma^2(x),\ldots,\sigma^{n-1}(x)$ are all different by Lemma \ref{shifts}, and a simplex code of length $n$ contains $n$ nonzero codewords, we conclude that
$$y\in\{x,\sigma(x),\sigma^2(x),\ldots,\sigma^{n-1}(x)\}.$$
Therefore, $C^\perp$ has only the weight $w=(c+2)2^{k-1}=2^{2k-1}$. In this case, $C$ has length $(c+3)n=(2^k+1)(2^k-1)=2^{2k}-1$. Since the minimum distance is $3$ and the dimension is $2k$, $C$ is a Hamming code.
\end{proof}

\begin{proposition}\label{pesminim2}
The number of codewords in $C$ of weight $3$ is:
$$
|C_3|=\frac{(c+3)n(q-1)}{6}\left[(n-1)(q-1)+(c+1)(c+2)\right].
$$
\end{proposition}

\begin{proof}
We compute separately the number of codewords in $C_3$ for the different possible cases.
\begin{itemize}
\item[a)] Codewords in $C_3$ with the three nonzero coordinates in $B_3 \cup \cdots \cup B_{c+3}$. We can apply here the arguments of Proposition \ref{pesminim} for $c+1$ instead of $c$. The result is:
\begin{equation}\label{eqa}
\frac{(c+1)n(q-1)}{6}\left[(q-1)(n-1)+c(c-1)\right].
\end{equation}
\item[b)] Codewords in $C_3$ with the three nonzero coordinates in $B_1 \cup B_2$. Clearly, all the nonzero coordinates must be in $B_1$ or in $B_2$. Since the triples in $B_1$ (or $B_2$) form a Steiner system (Theorem \ref{designs}), we have that this number of codewords is:
\begin{equation}\label{eqb}
2\frac{(q-1)^2n(n-1)}{6}=\frac{(q-1)^2n(n-1)}{3}.
\end{equation}
\item[c)] Codewords in $C_3$ with exactly one nonzero coordinate in $B_3 \cup \cdots \cup B_{c+3}$. Consider any column $\bh_i$ of $H^{(c)}$ in $B_3 \cup \cdots \cup B_{c+3}$. It is clear that there is exactly one column $\bh_j$ in $B_1$ and one column $\bh_\ell$ in $B_2$, such that $\bh_i$, $\bh_j$ and $\bh_\ell$ are linearly dependent. Hence, in this case we have exactly one codeword for each coordinate (and its multiples) in $B_3 \cup \cdots \cup B_{c+3}$. Thus, the result is:
\begin{equation}\label{eqc}
(q-1)n(c+1).
\end{equation}
\item[d)] Codewords in $C_3$ with exactly one nonzero coordinate in $B_1 \cup B_2$. The remaining pair of nonzero coordinates cannot be in the same block. Indeed, one such pair of coordinates is already covered by a triple in the same block. The corresponding columns of $H^{(c)}$ must have equal (up to multiples) the first $k$ coordinates or the second $k$ coordinates (depending on the given nonzero coordinate is either in $B_2$ or in $B_1$, respectively). Hence, for any pair of blocks in $\{B_3,\ldots,B_{c+3}\}$ we can choose $n$ columns (and their $q-1$ multiples) of one of these two blocks and we have two possibilities for the other block. The result is:
\begin{equation}\label{eqd}
2n(q-1)\binom{c+1}{2}=(c+1)cn(q-1).
\end{equation}
\end{itemize}

Adding (\ref{eqa}), (\ref{eqb}), (\ref{eqc}), and (\ref{eqd}), we obtain the statement.
\end{proof}

\begin{corollary}\label{CR1}
$\mbox{ }$
\begin{enumerate}
\item[(i)] For $q=2$ and $c=n-1$, $C$ is a binary Hamming code of length $2^{2k}-1$.
\item[(ii)] For $q> 2$ or $c < n-1$,  $C$ is a
linear completely regular $[(c+3)n,(c+3)n-2k,3;2]_q$ code
with intersection array
$$ \IA = \{(c+3)n(q-1), (c+2)\left((q-1)n-1-c\right); 1, (c+2)(c+3)\}.$$
\end{enumerate}
\end{corollary}

\begin{proof}
(i) It is already proved in Proposition \ref{pesos2}.

(ii) The length, dimension and minimum distance of $C$ are clear. By Proposition \ref{pesos2}, $C$ has external distance $s=2$. Since $C$ is not perfect, we have that $\rho > 1$ and, by Theorem \ref{params} (i), the covering radius is $\rho=2$. Hence, by Theorem \ref{params} (iii), $C$ is a quasi-perfect uniformly packed code.

The values of the intersection numbers $b_0=(c+3)n(q-1)$ and $c_1=1$ are straightforward since $C$ has minimum distance $3$.

Now, we compute the intersection number $a_1$, that is, the number of neighbors in $C(1)$ of any vector $z\in C(1)$. Without loss of generality, assume that $z$ is a one-weight vector. Then, $a_1$ is the addition of the number of two-weight vectors covering $z$ and covered by some codeword of weight $3$, and the $q-2$ vectors of weight $1$ at distance $1$ from $z$. Since the set $C_3$ of codewords of weight $3$ defines a $q$-ary 1-design (Theorem \ref{designs}), we have that
\begin{equation}\label{disseny1}
3|C_3|=(q-1)cnr,
\end{equation}
where $r$ is the replication number, i.e. the number of codewords in $C_3$ covering $z$. Of course, any such codeword covers two vectors of weight $2$ that, also, cover $z$. Thus, we have that $a_1=2r+q-2$. Combining with (\ref{disseny1}), we obtain
$$
a_1=\frac{6|C_3|}{(c+3)n(q-1)}+q-2,
$$
and substituting $|C_3|$ from Proposition \ref{pesminim2}, we get
$$
a_1=(q-1)n-1+(c+1)(c+2).
$$
Since $b_1=(c+3)(q-1)n-c_1-a_1$, we obtain
$$
b_1 = (c+3)(q-1)n-1-\left[(q-1)n-1+(c+1)(c+2)\right]=(c+2)((q-1)n-1-c).
$$
By Lemma \ref{graph},
$b_1|C(1)|=c_2|C(2)|$. Since $C$ has minimum distance $3$, we have $|C(1)|=(c+3)(q-1)n|C|$. Also, $|C(2)|=q^{(c+3)n}-|C(1)|-|C|$ because the covering radius of $C$ is $\rho=2$. Therefore, we can compute $c_2$:

$$
c_2 = \frac{b_1|C(1)|}{|C(2)|}=\frac{(c+2)((q-1)n-1-c)(c+3)(q-1)n|C|}{(q^{2k}-(c+3)(q-1)n-1)|C|}.
$$
Substituting $n=(q^k-1)/(q-1)$, the expression simplifies to $(c+2)(c+3)$.
\end{proof}

\begin{remark}
Almost all codes in the family given in Corollary \ref{CR1} are not completely
transitive. However, in the binary case, using computer calculations,
we conjecture that for any value of $k>2$, the codes of Corollary \ref{CR1}, are
completely transitive for $c\in \{2^k-5,2^k-4,2^k-3\}$.
\end{remark}

\begin{remark}
In the binary case, the extension of the codes given in Corollary \ref{CR1} are not completely regular
in almost all cases. However, for each value of $k$, there
are exactly two values of $c$ such that the obtained binary extended code
is completely regular as we show in the next theorem.
\end{remark}

Of course, for $q=2$ and $c=n-1$, the extended code is an extended Hamming code. Therefore, we consider the binary cases where $1\leq c \leq n-2$.

\begin{theorem}\label{estes}
For $1\leq c \leq n-2$ and $q=2$, the extended code $C^*$ is a completely regular code if and only if $c=2^{k-1}-2$. In this case, $C$ is a $[2^{k-1}(2^k+1),2^{k-1}(2^k+1)-1-2k,4;3]$ code with
$$
\IA=\{2^{k-1}(2^k+1),\;2^{k-1}(2^k+1)-1,\;2^{2k-2};\;1,\;2^{k-1}(2^{k-1}+1),\;2^{k-1}(2^k+1)\}.
$$
\end{theorem}

\begin{proof}
As can be seen in \cite[Prop. 1.1]{Bro2}, $C^*$ has covering radius $\rho^*=\rho+1=3$. Hence, if $C^*$ is completely regular, it must have external distance $s^*=3$. In other words, $(C^*)^\perp$ must have exactly $3$ nonzero weights (Theorem \ref{params} (iv)). A generator matrix for $(C^*)^\perp$ is obtained adding, first the zero column, and second the all-one row to the matrix $H^{(c)}$. Therefore, $(C^*)^\perp$ has at least the nonzero weights $w=(c+3)n+1$, $w_1=(c+3)2^{k-1}$ and $w_2=(c+2)2^{k-1}$ (see Proposition \ref{pesos2}). If $(C^*)^\perp$ has exactly these weights, then it is clear that $w_1+w_2=w$. This condition is equivalent to
$$
(2c+5)2^{k-1}=(c+3)(2^k-1)+1,
$$
which implies $c=2^{k-1}-2$.

In that case, $C^*$ is a $[2^{k-1}(2^k+1),2^{k-1}(2^k+1)-1-2k,4;3]$ code and, by Theorem \ref{params} (vi), we have that $C^*$ is a completely regular code.

Now, we compute the intersection numbers. Let $N=2^{k-1}(2^k+1)$ be the length of $C^*$. Since the minimum distance in $C^*$ is $4$, it is clear that $b_0=N$ and $c_1=1$. Moreover, giving a one-weight vector $z\in C^*(1)$, all its neighbors, except the zero codeword, are vectors in $C^*(2)$. Thus, $b_1=N-1$. For a vector $y\in C^*(3)$, we have that all its neighbors are in $C^*(2)$. Hence, $c_3=N$.

By Proposition \ref{estesos}, we have that $4|C^*_4|=N|C_3|$. From this and Proposition \ref{pesminim2}, we obtain:
\begin{eqnarray}\label{pes4}
|C^*_4| &=& \frac{2^{k-1}(2^k+1)(2^{k-1}+1)(2^k-1)}{4\cdot 6}\left[2^k-2+2^{k-1}(2^{k-1}-1)\right]\nonumber \\
        &=& \frac{2^{k-1}(2^{2k-2}-1)(2^{2k}-1)(2^{k-2}+1)}{12}.
\end{eqnarray}

By Theorem \ref{designs}, the set $C^*_4$ defines a $2$-$(N,4,\lambda)$-design. Using Proposition \ref{blocks} and (\ref{pes4}), we can compute the parameter $\lambda$:
\begin{eqnarray}
|C^*_4|=\lambda\frac{N(N-1)}{12}\;\Longrightarrow\;\lambda &=& \frac{2^{k-1}(2^{2k-2}-1)(2^{2k}-1)(2^{k-2}+1)}{2^{k-1}(2^k+1)\left(2^{k-1}(2^k+1)-1\right)}\nonumber \\
                                                           &=& \frac{(2^{2k-2}-1)(2^k-1)(2^{k-2}+1)}{2^{k-1}(2^k+1)-1}.\label{lambda}
\end{eqnarray}

Let $y\in C^*(2)$, without loss of generality, we can assume that $y$ has weight $2$. Then, $y$ is at distance two of exactly $\lambda+1$ codewords in $C^*$. Since any codeword has $\binom{N}{2}$ vectors at distance $2$, and all such vectors are in $C^*(2)$, we have the relation:
\begin{equation}\label{C2}
|C^*|\frac{N(N-1)}{2}=(\lambda+1)|C^*(2)|.
\end{equation}
Alternatively, (\ref{C2}) can be obtained counting in two ways the number of edges of the bipartite graph that has $C^*\cup C^*(2)$ as set of vertices, and a vertex in $C^*$ is adjacent to a vertex in $C^*(2)$ if the corresponding vectors are at distance two. Now, using (\ref{lambda}) and (\ref{C2}), we can compute $|C^*(2)|$:
\begin{eqnarray}\label{calculC2}
|C^*(2)| &=& |C^*|\frac{2^{k-2}(2^k+1)\left[2^{k-1}(2^k+1)-1\right]\left[2^{k-1}(2^k+1)-1\right]}{(2^{2k-2}-1)(2^k-1)(2^{k-2}+1)+2^{k-1}(2^k+1)-1}\nonumber \\
         &=& |C^*|\frac{2^{k-2}(2^k+1)(2^{4k-2}+2^{3k-1}-2^{2k}+2^{2k-2}-2^k+1)}{2^{k-2}(2^{3k-2}-2^{2k-2}+2^{2k}-1)}\nonumber \\
         &=& |C^*|(2^k+1)(2^k-1)=|C^*|(2^{2k}-1).
\end{eqnarray}
Next, we compute $|C^*(3)|= 2^N-|C^*|-|C^*(1)|-|C^*(2)|$. Clearly, $|C^*(1)|=N|C^*|$. Therefore, using (\ref{calculC2}), we obtain:
\begin{equation}\label{C3}
|C^*(3)|=|C^*|\left[2^{2k+1}-1-2^{k-1}(2^k+1)-(2^{2k}-1)\right]=|C^*|2^{k-1}(2^k-1).
\end{equation}

By Lemma \ref{graph}, we have that $b_1|C^*(1)|=c_2|C^*(2)|$ and $b_2|C^*(2)|=c_3|C^*(3)|$. Using these relations, (\ref{calculC2}), and (\ref{C3}), we obtain:
\begin{eqnarray*}
c_2 &=& \frac{N(N-1)}{(2^k+1)(2^k-1)}=\frac{2^{k-1}(2^k+1)(2^{2k-1}+2^{k-1}-1)}{(2^k+1)(2^k-1)}=2^{k-1}(2^{k-1}+1);\\
b_2 &=& \frac{2^{k-1}(2^k+1)2^{k-1}(2^k-1)}{(2^k+1)(2^k-1)}=2^{2k-2}.
\end{eqnarray*}
The statement is proved.
\end{proof}

\section{Sporadic completely regular codes by concatenation}

We have computationally checked that the following codes are completely regular with the specified parameters.

\begin{enumerate}
\item Let $C$ be the binary $[15,9,3;3]$-code with parity check matrix
\begin{equation*}
H = \left[
\begin{array}{ccccc}
\;K\;&0\;&0\;&K\;&K\;\\
\;0\;&K\;&0\;&K\;&K_1\;\\
\;0\;&0\;&K\;&K\;&K_2\;
\end{array}
\right],\;\;\mbox{where}\;\;
K = \left[
\begin{array}{ccc}
\;1\;&0\;&1\;\\
\;0\;&1\;&1\;
\end{array}
\right]
\end{equation*}
and $K_1$ (respectively, $K_2$) is obtained by one cyclic shift of
the columns of $K$ in one position (respectively, by two cyclic shifts). Then, $C$ is
 a $[15,9,3;3]$ completely regular code with
$\IA = \{15,12,1; 1,4,15\}.$
The binary $[16,9,4;4]$-code, obtained by extension of $C$, is completely regular with
$$\IA = \{16,15,12,1; 1,4,15,16\}.$$

\item Denote by $D(u,q)$ a difference matrix \cite{Beth}, i.e. a
square matrix of the order $qu$ over an additive group of order
$q$, such that the component-wise difference of any two rows
contains any element of the group exactly $u$ times.

Take the difference matrix
$D(2,3)$
\[
D=\left[
\begin{array}{cccccc}
0~&~0~&~0~&~0~&~0&~0\\
0~&0~&~1~&~2~&~2~&~1\\
0~&1~&~0~&~1~&~2~&~2\\
0~&2~&~1~&~0~&~1~&~2\\
0~&2~&~2~&~1~&~0~&~1\\
0~&1~&~2~&~2~&~1~&~0\\
\end{array}
\right]
\]

Let $H$ be a binary $(12 \times 18)$ matrix obtained from $D$ by
changing any element $i$ by the matrix $K_i$. Then, the
$[18,12,3;2]$ code with parity check matrix $H$ is a completely regular code
with
$\IA = \{18,15;1,6\}.$

\item Do the same construction as in item 2 for the matrix $D^*$, which is the difference matrix $D(2,3)$
without the trivial column.
The resulting $[15,9,3;3]$ code is CR with $\IA = \{15,12,1;1,4,15\}$. This code coincides with the code described in item 1.
\end{enumerate}

\section*{Acknowledgements}
This work has been partially supported by the Spanish grants TIN2016-77918-P, AEI/FEDER, UE., MTM2015-69138-REDT; the Catalan AGAUR grant 2014SGR-691 and
also by Russian fund of fundamental researches (15-01-08051).
\bibliographystyle{AIMS}

\end{document}